\documentclass[reqno,10pt]{amsart}

\usepackage[T1]{fontenc}  
\usepackage[latin1]{inputenc} 
\usepackage{setspace}                   
\usepackage{mathdots}
\usepackage{indentfirst}   
\usepackage{courier}     
\usepackage{type1cm}                    
\usepackage{listings}                   
\usepackage{titletoc}
\usepackage{amsmath,amscd,amsthm}
\usepackage{amsfonts}   
\usepackage{amssymb}
\usepackage[all]{xy}
\usepackage{fullpage} 

\newcommand{\finde}{  {\hfill $\square$}  \vskip0.5cm }

\newtheorem{theorem}{Theorem}[section]                                          
\newtheorem{prop}[theorem]{Proposition}                          
\newtheorem{lemma}[theorem]{Lemma}
\newtheorem{cor}[theorem]{Corollary}
\newtheorem{example}[theorem]{Example}

\newtheorem{definition}[theorem]{Definition}

\def\FF{{\mathcal F}}
\def\GG{{\mathcal G}}

\def\QQ{{\mathcal Q}}
\def\XX{{\mathcal X}}
\def\YY{{\mathcal Y}}
\def\CC{{\mathcal C}}

\def\RR{{\mathcal R}}
\def\PP{{\mathcal P}}
\def\TT{{\mathcal T}}
\def\NN{{\mathcal N}}
\def\fd{{\, \rightarrow\, }}
\def\lfe{{\, \longleftarrow\, }}

\def\lfd{{\, \longrightarrow\, }}
\def\trian{{\bigtriangleup}}

\def\nada {{\, \rule{0.4cm}{0.5pt}\,}}

\def\Hom {\mathop {\rm Hom} \nolimits}
\def\Ext {\mathop {\rm Ext} \nolimits}

\def\End {\mathop {\rm End}\nolimits}
\def\add {\mathop {\rm add}\nolimits}

\def\ind {\mathop {\rm ind }\nolimits}
\def\supp {\mathop {\rm supp}\nolimits}
\def\Supp {\mathop {\rm Supp}\nolimits}
\def\mod {\mathop {\rm mod}\nolimits}

\def\pd {\mathop {\rm pd}\nolimits}

\def\id {\mathop {\rm id}\nolimits}

\def\rank {\mathop {\rm rank}\nolimits}

\makeatletter
\@addtoreset{equation}{section}
\makeatother

\renewcommand\thefigure{\thesection.\@arabic\c@figure}
\renewcommand\thetable{\thesection.\@arabic\c@table}

\bibdata{prob}
\bibstyle{alpha} 

\title[ Stratifying systems 
over hereditary algebras]{Stratifying systems \\
over hereditary algebras}

\author{Paula A. Cadavid and Eduardo do N. Marcos}

\address{
USP, Instituto de Matemática e Estatística
\newline
Rua de Matão 1010, Cidade Universitaria, São Paulo-SP
\newline
e-mail:  paca@ime.usp.br
\newline
e-mail:  enmarcos@ime.usp.br
}

\thanks{}

\begin{document}
 

\begin{abstract}
This paper deals with stratifying systems over hereditary algebras. In the case of tame hereditary algebras we obtain a bound for the size
of the stratifying systems composed only by regular modules and we conclude that stratifying systems can  not be complete. For wild hereditary 
algebras with more than $2$ vertices we show that there exists a complete stratifying system whose elements are regular modules. In the other 
case, we conclude that there are no stratifing system over them with regular modules.  In one  example we built all the stratifying systems, with 
a specific form, having maximal number of regular summads.
\end{abstract}

\maketitle

\section{Introduction} 
Let $A$ be a finite dimensional $K$-algebra. Let $S_1$, $S_2$, $\ldots$, $S_n$  be a complete list of all non-isomorphic 
simple $A$-modules and we fix this ordering of simple modules. Let $P_i$ be the projective cover of $S_i$. With this order 
of simple modules we define for each $i$ the standard module $\trian_{i}$ to be the maximal quotient of $P_i$ with composition 
factors $S_j$ with $j \leq i$. The concept of stratifying system ($s.s.$) was introduced as a generalization of the standard 
modules in (Erdmann and Sáenz, 2003). Later, E. Marcos, O. Mendoza and C. Sáenz introduced in (Marcos et al., 2004) and 
(Marcos et al., 2005) the notion  of $s.s.$ via relative simple modules and via relative projective modules, respectively. 
In these works, the authors show that these concepts are equivalent.

On the other hand, the concept of excepcional sequence was introduced in (Gordentsev and Rudakov, 1987) with the aim of to 
study vector bundles on $ \mathbb{P}^2 $. Such sequences have been also used by several authors in the study of derived 
categories of algebraic varieties (see for instance (Bondal and Kapranov, 1989), (Rudakov, 1990)). The corresponding notion 
in the context of representations of a quiver is considered in (Crawley-Boevey, 1993). We observe that over a hereditary 
algebra, the concept of $s.s.$ and exceptional sequence are equivalent.

The goal of this paper is to study $s.s.$ over hereditary algebras. We introduce the notion of complete stratifying system ($c.s.s.$) 
which is a $s.s.$ of maximal possible size and we discuss the existence of a $c.s.s.$ consisting only of regular modules for a 
hereditary algebra. Every $s.s$ can be extended to a complete one. We calculate all the $c.s.s.$ over the Kronecker algebra 
and generalized Kronecker algebras. In addition, we construct a family of $c.s.s.$ with a maximal number of regular elements 
over the canonical algebra $\trian(\widetilde{\mathbb {A}}_{p,q}) $.

The paper is organized as follows: We start  Section 2 with the main definitions and notations. Section 3 is devoted to the 
study of the relationship between stratifying systems and tilting modules, using this, we introduce the notion of complete 
stratifying system.  In Section 4 we show give a bound on the size  $s.s.$ consisting of regular modules, over a concealed 
algebra of Euclidean type. In section 5 we show that for any wild hereditary algebra with more than $2$ vertices there exists 
a complete stratifying system whose elements are regular modules. So we can conclude that an algebra, with more than two non 
isomorphic simple modules, has a $c.s.s$ consisting of regular modules if and only if it is wild. We also exibt in this section 
the construction of all the $c.s.s.$ over the Kronecker algebra and the generalized Kronecker algebras. In the last section, we 
construct a $s.s.$ of regular modules with maximal size $(F,G)$ over the canonical algebra $ \trian(\widetilde{\mathbb {A}}_{p,q})$, 
where $F,G$ are sets of simple regular modules. We also extend our construction to a $c.s.s.$ of the form $(X,F,G,Y)$.

\section{Preliminaries} 
\label{sec:intro}
The algebras to be considered in this paper are basic, connected and finite dimensional over an
algebraically closed field $K$. We denote by $\mod A$ the category whose objects are finitely generated rigth $A$-modules
and by $D$ the usual duality $\Hom_K(\nada, K): \mod A \lfd \mod A^{op}$. All the subcategories of $\mod A$ to be considered 
are full subcategories.

Let $\CC$ be a subcategory of $\mod A$. We say that $\CC$ is
\begin{enumerate}
 \item closed under extensions if for  every  short exact sequence  of $A$-modules $0\fd M_1  \fd M_2 \fd M_3 \fd 0$ with  
 $M_1,\,M_3 \in \CC$, also $M_2\in \CC$;
\item  closed under direct summands provided  for every $A$-module $M \in \CC$, any direct summand of $M$ belongs to $\CC$;
\item closed under kernels of epimorphisms if for every $A$-module epimorphism $f:M \fd N$ with $M, N \in \CC$, also $\ker(f)\in \CC$.
\end{enumerate}

Given a finite  set $X$ of modules, we denote  by $\FF(X)$ the subcategory of $\mod A$
containing the zero module and all modules which are filtered by modules in $X$. That is, 
a non zero module $M$ belongs to $\FF(X)$ when there is a finite chain 
\[0=M_0 \subseteq M_1 \subseteq \ldots \subseteq M_t=M\]
of submodules of $M$ such that $M_i/ M_{i-1}$ is isomorphic to a module in $X$, for all $i=1,2,\ldots,t.$ 
Other category related  with $\FF(X)$ is $\YY(X)$ whose objects are the modules $M  \in \mod  A \mbox{ such that }
\Ext^{1}_{A}(\nada,M)|_{\FF(X)}=0 .$

We introduce now the main notion in this work.

\begin{definition}(Erdmann and Sáenz, 2003) Let $X=\{X_i\}_{i=1}^{t}$ be a set of non-zero $A$-modules 
and ${\underline Y} = \{Y_i\}^{t}_{i=1}$ be
a set of indecomposable $A$-modules. The pair $(X,{\underline Y})$ is a {\bf stratifying system} $(s.s.)$ of size $t$,
if the following three conditions hold:
\begin{enumerate}
\item  $\Hom_{A}(X_j,X_i)=0$, for $j > i$.
\item  For each $i \in \{1, \ldots,t\}$ there is an exact sequence $0\fd X_i \fd Y_i \fd Z_i \fd 0, $
where $ Z_i \in \FF(X_1,\ldots, X_{i-1}).$
\item $\Ext^{1}_{A}( \nada ,Y )|_{\FF(X)}=0$, where $Y=\coprod_{i=1}^{t} Y_{i}$.
\end{enumerate}
\end{definition}

The modules in the set $X$ will be called relative simple modules  in
 $\FF(X)$. This nomenclature is justified by the fact that they are exactly the  modules which do  have not proper
 submodules in  the category $\FF(X)$. In order to give a characterization of stratifying systems depending only on 
 the set of relative simple modules we will make use of the following result.

\begin{prop}(Erdmann and Sáenz, 2003)\label{equivale} Given a set $X =\{X_i\}_{i=1}^{t}$
 of non-zero $A$-modules the following conditions are equivalent:
\begin{enumerate}
\item  There exists a set of indecomposable $A$-modules ${\underline Y }= \{Y_i\}_{i=1}^{t}$
 such that $(X, {\underline Y })$ is a stratifying system of size $t$.
\item  The set $X$ satisfies the following conditions:
\begin{enumerate}
\item $\Hom_{A}(X_j,X_i)=0$, for $j > i$.
\item $\Ext_{A}^{1}(X_j, X_i)=0,$ for $j \geq i$.   
\item $X_i$ is indecomposable, for all $i = 1, 2,\ldots,t$.
\end{enumerate}

\end{enumerate}
\end{prop}

As a consequence of the previous result we have the following alternative definition of stratifying
systems.
     
\begin{definition} (Marcos et al., 2004)\label{ss} A {\bf stratifying system} ($s.s.$) of size $t$ consist
  of a set  $X = \{X_i\}_{i=1}^{t}$ of indecomposable $A$-modules
  satisfying the following conditions: 
\begin{enumerate}
\item $\Hom_{A}(X_j,X_i)=0$, for $j > i$.
\item $\Ext_{A}^{1}(X_j, X_i)=0,$ for $j \geq i$ .   
\end{enumerate}
\end{definition}

Given a $s.s.$ $X$ and an  $A$-module $M \in \FF(X)$, we
 denote by $[M:X_i]$  the multiplicity of the $A$-module
 $X_i$ on a given filtration of $M$. The number  $[M:X_i]$ is
 well  defined,  that is,  it does not depend on a given filtration of $M$ in $X$ (see (Erdmann and Sáenz, 2003)).
 The category $\FF(X)$ is closed by direct summads, under  extensions and funtorially finite in $\mod A$, so it has (relative)
almost split sequences (see (Ringel, 1991)). On the other hand, it is  not necessarily closed under kernels  of epimorphisms 
(see (Marcos et al., 2004)).

Let $A$ be a hereditary $K$-algebra. An $A$-module $X$ is called {\bf exceptional}
if $\End_A(X)\cong K$ and $\Ext^{1}_{A}(X,X)=0$ and a sequence $X=(X_1, \ldots,X_t)$  is called {\bf exceptional} if  it is a sequence 
of exceptional modules satisfaying $\Hom_A(X_i,X_j)=0$, for $i>j$ and $\Ext^{1}_A(X_i,X_j)=0$ for $i \geq j.$ A sequence $X=(X_1,\ldots,X_n)$  
is said to be complete if $n$ is the number of isomorphism classes of simple 
$A$-modules.

We observe that if $A$ is a hereditary $K$-algebra, then the sequence of $A$-modules $(X_1,\ldots, X_t)$ 
is an exceptional sequence if, and only if, the set $X = \{X_i\}_{i=1}^{t}$ is a $s.s.$ over $A$.

\section{Stratifying systems over hereditary algebras and tilting
         modules} 
We recall that an $A$-module $T$ is called  {\bf partial tilting} if the following 
conditions hold:
\begin{enumerate} 
\item[T1.] $\pd T \leq 1$,
\item [T2.] $\Ext_A^{1}(T,T)=0.$
\end{enumerate}
 A partial tilting module $T$ is called {\bf tilting module} if it satisfies 
the following  additional condition:
\begin{enumerate}
\item [T3.] There exist a short exact sequence $0\fd A_A \fd T_1 \fd T_2 \fd 0$, with $T_1$, $T_2$  in
$\add T$, where $\add T$ is the subcategory of $\mod\,A$ whose objects are direct sums of direct summands of $T$.

\end{enumerate}
Finally, a tilting $A$-module $T$ is called {\bf basic tilting module} if the number of pairwise non-isomorphic 
indecomposable summands  of $T$ is the rank of $K_0(A)$. 

\begin{lemma}\label{incli}Let $Q$ be a finite, connected and acyclic  quiver with $n$ vertices, $A=KQ$ and 
 $X = \{X_i\}_{i=1}^{t}$ be a $s.s.$ over $A$. Then $\FF(X)\cap\YY(X)=\add T$, where $T$ is a basic partial 
tilting $A$-module, and $t\leq n$.
\end{lemma}
\begin{proof} According to Theorem 2.4  in (Marcos et al., 2004) we have that $\FF(X)\cap\YY(X)=\add Y$, where 
$Y$ is an $A$-module with $t$ pairwise non-isomorphic indecomposable direct summands and such that $\Ext^1_A(Y,Y)=0$. 
This means that $Y$ is a  partial tilting module and therefore we conclude that $t\leq n$.
\end{proof}
In the previous lemma the hypothesis that $A$ is a hereditary algebra is essential as the following example shows.
\begin{example}(Marcos et al., 2004) Let $A$ be  the path algebra of the quiver $3 \overset{\alpha}\lfd 1 \overset{\beta}\lfe 2\overset{\gamma}\lfe 4 $
bound by  $\gamma\beta$. Taking $X_1= S_1=P_1$, $X_2=P_2$,
$X_3=P_3$, $X_4=P_4=I_2$ and $X_5=S_4$ we have that $X=\{X_1,X_2,X_3,X_4,X_5\}$ is a $s.s.$ of size $5$.
\end{example}
This motivates the following definition.
\begin{definition} Let $Q$ be a finite, connected and acyclic  quiver with $n$ vertices and $A=KQ$.
A $s.s.$  $X = \{X_i\}_{i=1}^{t}$ over $A$ is called {\bf complete}  ($c.s.s.$) if $t=n.$
\end{definition}

Let $Q=(Q_0 ,Q_1)$ be a quiver such that $|Q_0|=n.$  An {\bf admissible numbering} of $Q$
is a bijection between $Q_0$ and the set $\{1,2, \ldots n\}$ such that, if  we have an 
arrow $j \fd i$ then $j>i$. Let us note that  a connected quiver $Q$, with more than one vertice, is an acyclic quiver if 
and only if  there exist an admissible numbering of $Q$.

\begin{prop}Let $A$ be a hereditary $K$-algebra and $T$ be a basic tilting $A$-module.
Then there exists a numbering of the sumands of $T=\coprod_{i=1}^{n}T_i$ such that 
$(T_1,\ldots, T_n)$ is a $c.s.s.$
\end{prop}

\begin{proof} Let $B=\End_A(T)$. Then the ordinary quiver $Q_B$ of $B$ is acyclic. If 
${\boldsymbol e}=(e_1,\ldots,e_n)$ is an admissible ordering of the vertices of $Q_B$ and we take
$P_i=e_{i}B$ then $\Hom_B(P_i,P_j)=0,$ for $j<i$. So $(P_1, \ldots, P_n)$ is a $s.s.$ over $B$. Since
$\Hom_A(T,\nada): \add T \lfd \PP$, when  $\PP$ is the category of projective $B$-modules,
is  an equivalence of categories then $\Hom_A(T_i,T_j)=0,$ for $j<i$. But $\Ext^1_{A}(T_i,T_j)=0$,
for all $i,j$, then we get that $(T_1, \ldots, T_n)$ is a $c.s.s.$
\end{proof}

\section{ s.s over hereditary standard stable tubes} 

We start this section by  collecting  some preliminary facts about hereditary standard stable tubes 
needed latter in this section. The facts which we describe next can be found  
in (Simsom and Skowronski, 2007, Vol. 2 and Vol. 3)  and we include here for the sake of completeness.  

We recalling that a standard tube $\TT$, of the Auslander-Reiten quiver of an algebra $A$, 
is hereditary when $\pd X \leq 1$ and $\id X \leq 1$, for any module $X$ of $\TT$. 

Let $A$ be  an algebra, $\TT_{\lambda}$ be a fixed  hereditary standard stable tube of rank $r_{\lambda}$ in
$\Gamma(\mod A)$. The set of modules having exactly  one immediate  predecessor in $\TT_\lambda$ 
is called the mouth of $\TT_\lambda$. Given a module $X$ lying on the mouth  of $\TT_\lambda$, a ray  starting at $X$ is 
defined to be the unique infinite sectional path $X=X[1]\fd X[2]\fd X[3]\fd \cdots \fd X[m]\fd \cdots $ in $\TT_\lambda$.

Let $E_1,E_2 \dots, E_{r_{\lambda}}$ be the $A$-modules lying on the mouth of $\TT_{\lambda}$
which form a $\tau_A$-cycle $(E_1,E_2 \dots, E_{r_{\lambda}})$, that is, they are ordered in such a way that
$\tau E_2 \cong E_1$, $\tau E_3\cong E_2, \ldots, \tau E_{r_{\lambda}} \cong  E_{r_{\lambda}-1},\tau E_1\cong E_{r_{\lambda}}. $ Then
\begin{itemize}
\item  the modules $E_1,E_2 \dots, E_{r_{\lambda}}$ are  pairwise  orthogonal bricks in $\mod A$;
\item every indecomposble $A$-module $M$ in $\TT_{\lambda}$ is of the form $M \cong E_j[m]$, for some $j\in\{1,\ldots,r_{\lambda}\}$
      and $m\geq 1$;
\item the subcategory $\add \TT_{\lambda}$ of $\mod A$ is hereditary, abelian  and the modules 
$E_1,E_2 \dots, E_{r_{\lambda}}$ form a complete  set of pairwise  non-isomorphic  simple objects of  $\add \TT_{\lambda}$;
\item the module $E_i=E_i[1]$ is the unique  simple subobject of $M\cong E_{i}[j]$, and $j$ is the length $\ell_{\lambda}(M)$ 
      of $X$ in $\add \TT_{\lambda}$.
\end{itemize}

For any indecomposable $A$-module $M\cong E_{i}[j]$ we denote by $\CC(M)$ the cone in the tube $\TT_{\lambda}$ 
determined by $M$. That is, $\CC(M)$ is the full translation subquiver of  $ \TT_{\lambda}$ of the form 
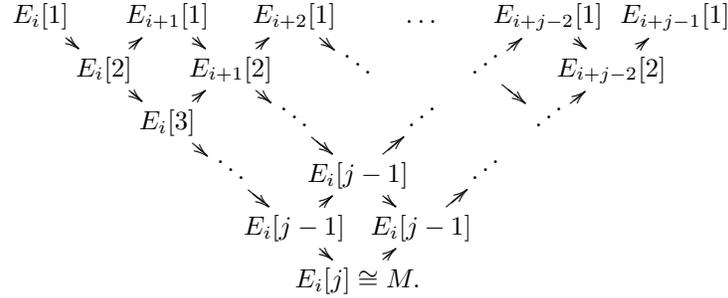
\begin{displaymath}
\xymatrix@!0 @R=2em @C=2pc{
E_i[1] \ar[dr] && E_{i+1}[1] \ar[dr] && E_{i+2}[1] \ar[dr] && \ldots && E_{i+j-2}[1] \ar[dr] && E_{i+j-1}[1]  \\
& E_i[2] \ar[dr]\ar[ur] && E_{i+1}[2] \ar[dr]\ar[ur] && \ddots && \iddots \ar[ur] \ar[dr] &&  E_{i+j-2}[2] \ar[ur]\\
& & E_i[3] \ar[ur]\ar[dr] &&  \ddots  \ar[dr] &&   \iddots &&  \iddots \ar[ur] \\
&&& \ddots \ar[dr] && E_{i}[j-1] \ar[dr] \ar[ur] && \iddots \\
&&&& E_i[j-1]  \ar[dr] \ar[ur] &&  E_i[j-1]  \ar[ur] \\
&&&&& E_{i}[j]\cong M.\ar[ur]\\
}
\end{displaymath}

\begin{lemma}\label{oba1}Let  $A$ be an algebra and $\TT_{\lambda}$ a hereditary standard stable tube of 
$\Gamma(\mod A)$ of rank $r_{\lambda} \geq 1$.  If $N_1,\ldots,N_t$ are parwise non-isomorphic indecomposable 
$A$-modules such that $\Ext^1_A(N_1\oplus\ldots\oplus N_t,N_1\oplus\ldots\oplus N_t)=0$ and $\CC(N_i)\cap\CC(N_j)=\emptyset$, for $i\not= j,$ then 
$ \ell_{\lambda}(N_{1}) + \cdots +\ell_{\lambda}(N_{t}) \leq r_{\lambda}-t.$
\end{lemma}
\begin{proof}Let $(X_1,\ldots, X_{r_{\lambda}})$ be a $\tau$-cycle of mouth modules of the tube
 $\TT_{\lambda}$. We denote by $\NN_i$ and $\NN_i^{*}$ the sets $\NN_i=\CC(\NN_i)\cap\{X_1,\ldots X_{r_{\lambda}}\}$
 and $\NN_i^{*}= \tau\CC(N_i)\cap \{X_1,\ldots,X_{r_{\lambda}}\}.$ If  $N_i \cong X_{l_i}[t_i]$, then
 $$\NN_i=\{ X_{l_i}[1],\ldots, X_{l_{i}+t_{i}-1}[1]\} \text{ and }\NN_i^{*}=\{ X_{l_{i}-1}[1],\ldots, X_{l_{i}+t_{i}-2}[1]\}.$$ 
 Therefore $|\NN_i \cup \NN_i^{*} |= t_i +1= \ell_{\lambda}(N_i)+1.$ It follows from the hypothesis and  from Lemma 1.7 in 
 (Simsom and Skowronski, 2007, Vol. 3, Chap. XVII) that  $\CC(N_i)\cap\tau\CC(N_j)=\emptyset$ and $\tau\CC(N_i) \cap \CC(N_j)=\emptyset$. Hence
$$\left|\bigcup_{i=1}^{t}\left(\NN_i \cup \NN_i^{*}\right)\right|=\left(\sum_{i=1}^{t} \ell_{\lambda}(N_i) \right)+t. $$
On the other hand, we have $ \bigcup_{i=1}^{t}\left(\NN_i \cup \NN_i^{*} \right) \subseteq \{X_1, \ldots, X_{r_{\lambda}}\}.$ Then
$\sum_{i=1}^{t} \ell_{\lambda}(N_i)\leq r_{\lambda}-t.$
\end{proof}
 
\begin{prop}\label{llegue2}Let  $A$ be an algebra and $\TT_{\lambda}$ a hereditary standard stable tube of 
$\Gamma(\mod A)$ of rank $r_{\lambda} \geq 1$. If $T$ is a multiplicity-free partial tilting $A$-module, then $T$ has 
at most $r_{\lambda}-1$  direct summands in $\add \TT_{\lambda}$.
\end{prop}

\begin{proof} Let $T$ be a  partial tilting $A$-module. We define sets $\XX$ and $\XX_{0}$ as follow
 \[
 \begin{array}{rcl}
 \XX &=&\{X  : X \in \add \TT_{\lambda} \cap\add T \cap  \ind A\}  \text{ and }\\
\XX_{0}&=&\{X \in \XX : \forall Y \in\XX, \CC(Y)\subseteq \CC(X)\text{ or } \CC(X)\cap \CC(Y)= \emptyset\}.
 \end{array}
 \]
Since $T$ is a multiplicity-free partial tilting $A$-module hence $\XX$ and $\XX_{0}$ are finite.
We note that if $X_1 ,X_2 \in \XX_0$ and $X_1  \not =X_2$, then $\CC(X_1) \cap\CC(X_2)=\emptyset. $ Therefore  
if $|\XX_0|=\ell$, by \ref{oba1}, we have 
 $\sum_{X \in \XX_0}\ell_{\lambda}(X) \leq r_{\lambda}-\ell. $
 Furthermore, if $X \in \XX_{0}$ then, by  1.6 (Simsom and Skowronski, 2007, Vol. 3, Chap. XVII), $| \CC(X) \cap \XX | \leq \ell_{\lambda}(X).$
Thus
 $$|\XX| \leq \sum_{X \in \XX_{0}} | \CC(X) \cap \XX| \leq \sum_{X \in \XX_0}\ell_{\lambda}(X) \leq r_{\lambda}-\ell  $$
and we conclude that if $\XX\not = \emptyset$ then $|\XX|\leq r_{\lambda}-1.$
  \end{proof}

\begin{prop}\label{regmax} Let $A$ be an algebra and $\TT_{\lambda}$ a hereditary standard stable tube of 
$\Gamma(\mod A)$ of rank $r_{\lambda} \geq 1$. Then there exist a $s.s.$ of size $r_{\lambda}-1$ lying
in  $\mathcal{T}_{\lambda}$. Further, if $X=\{X_{1},\ldots,X_{t}\}$ is a $s.s.$
such that $X_i \in \add \mathcal{T}_{\lambda}$, for all $i=1,\ldots, t$, then $t\leq r_{\lambda}-1$.
\end{prop}

\begin{proof}Let $(X_1, \ldots, X_{r_{\lambda}})$ be  a $\tau$-cycle of modules lying on the mouth of  $\TT_{\lambda}$.
Using the Auslander formula, we get  that if $1<i\leq j\leq r_{\lambda}-1$ then
$\Ext^1_A(X_i,X_j)\cong D\Hom_A(X_j, \tau X_i)\cong D\Hom_A(X_j,X_{i-1}).$
Since $j \not= i-1$ we see that  $\Hom_A(X_j,X_{i-1})=0$. Therefore, $\Ext^1_A(X_i,X_j)=0$. 
If $i=1$ and  $1\leq j\leq r-1$, again by the Auslander formula, we have 
$\Ext^1_A(X_1,X_j)\cong D\Hom_A(X_j,\tau X_1)\cong D\Hom_A(X_j,X_{r_{\lambda}}).$
Because $j\not = r$ it follows that $\Hom_A(X_j,X_{r})=0$. Thus $\Ext^1_A(X_1,X_j)=0$.
This implies that $X=(X_{r_{\lambda}-1},\ldots,X_{1})$ is a $s.s.$ of size $r_{\lambda}-1$. 

Now suppose that $Y=\{Y_{1},\ldots,Y_{t}\}$ is a $s.s.$ such that 
$Y_{i}\in \add \mathcal{T}_{\lambda}$. As $\add \TT_{\lambda}$ is  abelian and  extension closed subcategory 
of  $\mod A$ then $\FF(Y)\subseteq \add \mathcal{T}_{\lambda}.$
On the other hand, by \ref{incli}, $\FF(Y)\cap\YY(Y)=\add T$, where $T$ are a partial tilting $A$-module and then, by 
\ref{llegue2}, the number of indecomposable direct summands
of $T$ are at almost $r_{\lambda} -1$. Because the size of $X$ is equal to a number of indecomposable direct summands
of $T$ we have that $t \leq r_{\lambda}-1$.
\end{proof}

Let $Q$ be a finite, connected and acyclic quiver that is
not a Dynkin quiver. An algebra $B$ is called {\bf concealed of type Q } if there
exists a postprojective tilting module $T$ over the path algebra $A = KQ$ such
that $B =\End_A(T_A).$

Note that if $A=KQ$ and $Q$ is an Euclidean quiver then $A$ is a concealed algebra of type $Q$. We recall from 
(Simsom and Skowronski, 2007, Vol. 2) that, if  $B =\End_A(T_A)$ is a concealed algebra of Euclidean type then  
$\pd Z \leq 1$ and $\id Z \leq 1$, for all but finitely many non-isomorphic indecomposable $B$-modules $Z$ that are 
postprojective or preinjective. The category $\RR(B)$, whose objects are all the regular $B$-modules, is abelian and 
closed by extensions and the components of $\RR(B)$ are a family $\TT^{B}=\{\TT_{\lambda}^{B}\}_{\lambda \in \Lambda}$ 
of pairwise orthogonal standard stable and hereditary tubes. Moreover, if $r^{B}_{\lambda}\geq 1$ denotes the rank of 
$\TT^{B}_{\lambda}$ and $n$ is the rank of $K_0(A)$, then
$$\sum_{\lambda \in \Lambda} r^{B}_{\lambda}-1 \leq n-2.$$

Using the properties of concealed algebras mentioned above and \ref{regmax}
we have the following result.

\begin{theorem}\label{importante} Let $A$ be a concealed algebra of Euclidean type. If $X$ is a regular $s.s.$ of size 
$t$ over $A$ then $t \leq \rank K_{0}(A)-2$. In particular, if $A$ is a hereditary algebra of Euclidean type
there is no a regular  $c.s.s.$ over $A$.
\end{theorem}

\section{ s.s. over wild hereditary algebras}

In this section we describe all the $c.s.s.$ over the Kronecker algebra and the generalized Kronecker
algebras, and give a caracterization of wild hereditary algebras in terms of stratifying systems. 

We begin by stating some known statements, which we will use. 

\begin{lemma}(Crawley-Boevey, 1993)\label{estender} Any exceptional sequence $(X_{1},\ldots,X_{a},Z_{1},\ldots,Z_{c})$
can be enlarged to a complete 
sequence $(X_{1},\ldots, X_{a},Y_{1}, \ldots, Y_{b},Z_{1},\ldots, Z_{c}).$
\finde
\end{lemma}

\begin{lemma}(Crawley-Boevey, 1993)\label{unico} If $E=(X_{1},\dots,X_{n})$ and $F=(Y_{1},\dots,Y_n)$ are complete 
exceptional sequences such that  $X_j\cong Y_j$ for $j \not= i$, then $E=F$.
\finde
\end{lemma}

\begin{prop}\label{kro}Let $A=KQ$, where $Q$ is the quiver
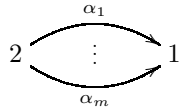
\begin{displaymath}
\xymatrix  @!0 @R=3em @C=5pc{
2 \ar@/^1pc/[r]^{\alpha_1}_{\vdots } \ar@/_1pc/[r]_{\alpha_m}& 1}
\end{displaymath}
with $m\geq 2$. Then  the list of all  $c.s.s.$ over $A$ is the following:
\begin{enumerate}
\item $(I_2, P_1)$.
\item  $(\tau^{-i}P_1, \tau^{-i}P_2)$, with $i\geq 0$.
\item $(\tau^{-i}P_2 , \tau^{-i-1}P_1)$, with $i\geq 0$.
\item $(\tau^{i}I_1 , \tau^{i}I_2)$, with $i\geq 0$.
\item  $(\tau^{i+1}I_2 , \tau^{i}I_1)$, with $i\geq 1$.
\end{enumerate}
\end{prop}
\begin{proof} First we observe that independently of $m$ all the indecomposable regular modules over $A$ have self-extensions. 
In fact, let $M$ be a regular indecomposable $A$-module. If $m=2$ then $M$ lives in an homogeneous tube of $\Gamma (\mod A)$, 
so it has self extension, at least the one given by the almost split sequence. If $m>2$, as a consequence of 2.16 in (Simsom 
and Skowronski, 2007, Vol. 3, Chap. XVII), $M$ has non trivial self-extensions. So all stratifying systems consist only of 
posprojetive or preinjetive modules.

Just as an example, we will show that the pair (2) in the list is a $s.s.$. The other verifications are similar. Using 
the Auslander formula we have the isomorphism $\Hom_A(\tau^{-i}P_2, \tau^{-i}P_1) \cong \Hom_A( P_2, P_1).$
But $P_1$ is a simple projetive module and then $\Hom_A( P_2, P_1)=0$. On the other hand, using the Auslander formula again, we have that
$\Ext^{1}_{A}(\tau^{-i}P_2,\tau^{-i}P_1) \cong D\Hom_A(\tau^{-i}P_1,\tau^{-i+1}P_2) 
                                         \cong D\Hom_A(\tau^{-1}P_1,P_2)
                                         =0.$
Therefore  $(\tau^{-i}P_1,\tau^{-i}P_2)$, with $i\geq 0$, is a $c.s.s.$ over $A$.

Finally we must guarantee that the list is complete. Observe that all posprojective and all preinjective appear as the first entry on exactly 
one of the pairs described by the list, so the result follow by \ref{estender}.
\end{proof}

\begin{cor}Let $A$ be as in the previous theorem and $X$ be a $c.s.s.$ over $A$. The category $\FF(X)$ is infinite if, and only if, $X = (I_2, P_1)$.
\end{cor}
\begin{proof} Since $I_2=S_2$ and $P_1=S_1$ then $\FF(I_2, P_1)=\mod A$.  In the other cases, by  using  Lemma \ref{repete} can be shown that
$\FF(X_1,X_2)=\add (X_1\oplus X_2)$.
\end{proof}

We recall that if $Q$ is a finite connected acyclic quiver and $A = KQ$ then there exists a regular tilting module in $\mod A$ 
if, and only if, $Q$ is neither a Dynkin  or a Euclidean quiver, and $Q$ has at least three vertices (see (Baer, 1989)). Therefore, 
by \ref{importante} and \ref{kro}  we have the main result of this section which is stated in the following theorem.

\begin{theorem} Let $Q$ be a finite connected acyclic quiver, $|Q_0|\geq 3$ and $A = KQ$.
There exists a regular  $c.s.s$ over $A$ if, and only if, $A$ is a wild hereditary algebra.
\end{theorem}
 
\section{Stratifying systems over algebras of type $\Delta(\widetilde{\mathbb{A}}_{p,q})$}
In this section  $p$ and $q$ are integers such that $1\leq p\leq q$. Let $\Delta=\Delta(\widetilde{\mathbb{A}}_{p,q})$
be the canonically  oriented  Euclidean quiver 
\begin{displaymath}
\xymatrix @!0 @R=3em @C=4pc{
& &  \ar[dl]  \stackrel{1}{\circ}& \stackrel{2}{\circ} \ar[l] & \ar[l]  \cdots & \ar[l] \stackrel{p-1}{\circ} \\
\Delta(\widetilde{\mathbb{A}}_{p,q}): & _{0}\circ    &           &       &     &&  \ar[ul] \ar[dl] \,\,\,\,\circ_{p+q-1}\\
& &  \ar[ul]  \stackrel{p}{\circ}&   \stackrel{p+1}{\circ}\ar[l] & \ar[l]  \cdots &  \stackrel{p+q-2}{\circ}\ar[l] 
  }
\end{displaymath}

We give a  construction of $s.s$ over $K\Delta$ with a maximal number of regular modules.  To start we need to recall the description of the 
simple regular representations of $K\Delta$, and their behavior under $\tau$,  which we give next.

\noindent{\bf { (a) Simple regular  representations in the tube $\mathcal{T}^{\Delta}_{\infty}$}}
 $$E_{i}^{(\infty)}=S_{i},\text{ com } 1\leq i \leq p-1,$$
 
\xymatrix  @!0 @R=2em @C=3pc {
                                                      & & & & &  \ar[dl] 0 & 0 \ar[l] & \ar[l]  \cdots & \ar[l] 0 \\
& &&E_{p}^{(\infty)}:& K &  &  & & &  \ar[ul] \ar[dl]^{1} K\\
                                                      & & && &  \ar[ul]^{1} K & K \ar[l]^{1} & \ar[l]^{1}  \cdots & K \ar[l]^{1}
  }

with the following properties: 
\begin{equation}\label{propriedades1}\tau E_{i+1}^{(\infty)}=E_{i}^{(\infty)},\text{ for } 1 \leq i \leq p-1, \,\,\,\,\,
\text{ and  } \,\,\,\,\,\tau E_{1}^{(\infty)}=E_{p}^{(\infty)}.
\end{equation}
  
\noindent {\bf { (b) Simple regular representations in the tube  $ \mathcal{T}^{\Delta }$}}
$$E_{j}^{(0)}=S_{p+j-1},\text{ for } 1\leq j \leq q-1$$
$$\xymatrix @!0 @R=2em @C=3pc {
 &&& &&   \ar[dl]_{1} K & K \ar[l]_{1} & \ar[l]_{1}  \cdots & \ar[l]_{1} K \\
  &&&E_{q}^{(0)} :&K &  &  & & &  \ar[ul]_{1} \ar[dl] K,\\
 &&&& &  \ar[ul] 0 & 0 \ar[l] & \ar[l] \cdots & 0 \ar[l]\\
  }$$
  
 with the following properties: 
\begin{equation}\label{propriedades2}
\tau E_{j+1}^{0}=E_{j}^{0},\text{ for } 1 \leq j \leq q-1,\,\,\,\,\,\text{ and}\,\,\,\,\,\tau E_{1}^{0}=E_{q}^{0}.
\end{equation}

\noindent{\bf{(c) Simple regular representations in the tube  $\mathcal{T}^{\Delta}_{\lambda}$, 
with $\lambda \in K \setminus \{0\}$}}\newline
\begin{displaymath}
\xymatrix @!0 @R=2em @C=3pc {
  & &  \ar[dl]_{\lambda} K & K \ar[l]_{1} & \ar[l]_{1}  \cdots & \ar[l]_{1} K \\
 E^{(\lambda)} :&K &  &  & & &  \ar[ul]_{1} \ar[dl]_{1} K\\
 & &  \ar[ul]_{1} K & K \ar[l]_{1} & \ar[l]_{1} \cdots & K \ar[l]_{1}\\
  }
\end{displaymath}
\vspace{0.3cm}

The following theorem  give us information about the tubular components of $\Gamma (\mod K\Delta)$. We will use 
it in order to construct the $s.s.$ of our interest.

\begin{theorem}(2.5 in Simsom and Skowronski, 2007, Vol. 2, Chap. XII) Assume that  $A=K\Delta$. Then
every component in the regular  part $\RR(A)$ of $\Gamma(\mod A)$ is one of the following stable tubes:
\begin{enumerate}
\item the tube  $\mathcal{T}^{\Delta}_{\infty}$ of rank $p$ containing the $A$-modules 
 $E_{1}^{(\infty)},\ldots,E_{p}^{(\infty)}$.
\item the tube $\mathcal{T}^{\Delta}_{0}$ of rank $q$
containing the $A$-modules  $E_{1}^{(0)}, \ldots, E_{q}^{(0)}$.
\item  the tube  $\mathcal{T}^{\Delta}_{\lambda}$ of rank $1$
containing the $A$-modules  $E^{(\lambda)}$, with $\lambda\in K\setminus \{0\}.$
\end{enumerate}
Where  $E_{j}^{(\infty)}$, $ E_{i}^{(0)}$ and  $E^{(\lambda)}$ are the simple regular  $A$-modules previously defined.
\end{theorem}

We use the simple regular modules in the tubes $\mathcal{T}^{\Delta}_{\infty}$ and
$\mathcal{T}^{\Delta}_{0}$ to build a $s.s.$ of  size $p+q-2$ over $K\Delta$.
Let $F_i$ denote the module $E_{p-i}^{(\infty)}$, for $i=1,\ldots,p-1$, and let $G_i$ denote the module $E_{q-i}^{(0)}$, for $i=1,\ldots,q-1$. 
Then we have that
$$\tau F_i=F_{i+1}, \text{ for } 1 \leq i \leq p-2, \tau F_{p-1}=E_{p}^{(\infty)} \text{ and } \tau E_p= F_1 ,$$
$$\tau G_i= G_{i+1}, \text{ for  } 1 \leq i \leq q-2, \tau E_q=G_{1} \text{ and }  \tau E_q=G_{1}.$$

The sequence  $(F_1, \ldots,F_{p-1})$ will be denoted by $F$ and the set 
$\{F_1, \ldots,F_{p-1}\}$ by $\FF.$ Analogously, $(G_1, \ldots,G_{q-1})$ will be denoted by $G$
and $\{G_1,\ldots, G_{q-1} \}$ by $\GG$. We claim that $F$ and $G$ are $s.s.$.

Let $S$ be a set of modules and let $M$ be a module. For simplicity  we write $\Hom_A(S,M)=0$ when $\Hom_A(X,M)=0$, 
for all $X \in S$. If $S$ is a set of regular modules, we denote  by $\tau^{i}(S)$, with $i\in\mathbb{Z}, i\not= 0$, the set 
$\tau^{i}(S)=\{\tau^{i} X, X \in S\},$ where $\tau$ is the Auslander-Reiten translation.

 Given a finite dimensional $K$-algebra, let $\{S_1, \ldots, S_n\}$ be a complete  set of 
the isomorphism classes of simple  modules, let $M$ be a module and let $\XX$ be a set 
of modules. The support of $M$ is the set $\supp\,M= \{i \in \{1, \ldots, n\}: [M:S_i] \neq 0\}.$
The support of $\XX$ is the set $\Supp\, \XX=\{i \in \{1, \ldots, n\}: [M:S_i] \neq 0, M \in \XX\}.$

\begin{lemma}\label{suporte}Let  $n\in \mathbb{N}$, $n\not =0$. Then:
\begin{enumerate}
\item 
      \begin{enumerate}
       \item If $p|n$, then $\Supp \tau^n \FF = \{1, \ldots, p-1 \}=\Supp \FF.$
       \item If   $n\equiv r \,\,(\mod \,p)$ and $ 1 \leq r \leq p-1$, then 
           \[ \begin{array}{lll} 
           \Supp \tau^{n}\FF =\tau^{r}\FF&=&\{0,\ldots, p+q-1 \} \setminus \{p-r\}\\
           \Supp \tau^{-n}\FF= \tau^{-r}\FF  &=& \{0,\ldots,p+q-1 \} \setminus \{r\}.\\
           \end{array}\]
       \end{enumerate}
\item  
       \begin{enumerate}
       \item If $q|n$, then $\Supp \tau^n \GG = \GG= \{p, \ldots, p+q-2 \}.$
       \item If  $n\equiv r \,\,(\mod \,q)$ and $ 1 \leq r \leq q-1$, then 
        \[  \begin{array}{rcl}    
        \Supp \tau^{n}\GG =\tau^{r}\GG&=& \{0, \ldots, p+q-1 \} \setminus \{p+q-r-1\}\\
        \Supp \tau^{-n}\GG=\tau^{-r}\GG &=& \{ 0,\ldots,p+q-1 \} \setminus \{p+r-1\}.\\
      \end{array}\]
     \end{enumerate}
 \end{enumerate}
 \end{lemma}
\begin{lemma}\label{repete}Given a hereditary algebra. 
\begin{enumerate} 
\item If $P_j$ and $P_m$ are  projective indecomposable modules then
$\Ext^1(\tau^{-t} P_j, \tau^{-t-r} P_m)=0,\text{ for all } t, \,r\geq 0.$
\item If $ I_j$ and $I_m$ are injective indecomposable $A$-modules then 
     $\Ext^1(\tau^t I_j, \ \tau^{t-r} I_m)=0,\text{ for all } t\geq r\geq 0.$
\end{enumerate}
\end{lemma}
      
\begin{proof} We just prove the statement (1) because the proof of (2) is similar. By the Auslander
formula we have that
$\Ext^{1}(\tau^{-t}P_{j},\tau^{-t-r}P_{m}) \cong  D\Hom(\tau^{-r-1}P_{m},P_{j}).$
But $\Hom(\tau^{-r-1}P_{m},P_{j})=0$, otherwise $P_{j}$ would have a non-projective predecessor in 
$\Gamma(\mod A)$.
\end{proof}

\begin{prop}\label{posprojetivo1} If $Y$ is a posprojective $A$-module such that  $(F,G,Y)$ is a $s.s.$, then 
$Y$ is one of the modules in the following list:
\begin{enumerate}
\item $P_0$.
\item $P_{p+q-1}$.
\item $\tau^{-t}P_0$, with $t\geq 1$ such that  $p|t$ and $q|t$.
\item $\tau^{-t} P_{p+q-1}$, with  $t\geq 1$ such that $p|t$ and $q|t$.
\item $\tau^{-t}P_{p-r}$, with $t\geq 1$ such that $q|t$, $t\equiv r \,\,(\mod \,p)$ and $ 1\leq r \leq p-1$. 
\item $\tau^{-t}P_{p+q-r-1}$, with $t\geq 1$ such that $p|t$, $t\equiv r \,\,(\mod \,q)$ and $1 \leq r \leq q-1$.
\end{enumerate}
\end{prop}

\begin{proof} If $M$ is a regular module, then
$\Ext^{1}_{A}(\tau^{-t}P_j,M)\cong D\Hom_A(\tau^{-1}M,\tau^{-t}P_j)=0, \text{ for } t\geq 0.$
Because of that, to prove that  $(F,G,\tau^{-t}P_j)$ is a $s.s.$ we need to show that 
$\Hom_A(\tau^{-t}P_j,\FF)=0= \Hom_A(\tau^{-t}P_j, \GG).$
On the other hand, as a consequence of Auslander formula, we have that
$$\Hom_A(\tau^{-t}P_j, \FF)\cong \Hom_A(P_j, \tau^{t}\FF) \text{ and } \Hom_A(\tau^{-t}P_j,\GG) \cong \Hom_A(P_j,\tau^{t}\GG).$$
Therefore,
$$\Hom_A(P_j, \tau^{t}\FF)=0 \text{ if, and only if, } j \notin \Supp\tau^{t}\FF \text{ and }$$
$$\Hom_A(P_j, \tau^{t}\GG)=0 \text { if, and only if,   }j \notin \Supp\tau^{t}\GG.$$
It follows that 
$(F,G,\tau^{-t}P_j)\text{  is a } s.s. \text { if, and only if,  } j \in (\Supp\tau^{t}\GG)' \cap (\Supp\tau^{t}\FF)',$  where 
$(\Supp\tau^{t}\GG)'$ denotes the complement of the set
$(\Supp\tau^{t}\GG)$ with relation to the set $\{ 0, \ldots, p+q-1\}$.
\vspace{0.1cm}

\noindent {\bf Case 1.} If $t=0$, then $Y\cong P_j$ and $(F,G,P_j)$ is a $s.s.$ if, and only if, 
$j \in (\Supp\,\FF)' \cap (\Supp\,\GG)'=\{0, p+q-1 \}$.

\noindent {\bf Case 2.} If $t$ is such that $p|t$ e $q|t$ then, according to \ref{suporte},
$\Supp\,\tau^{t}\FF= \Supp\,\FF$ and $\Supp\,\tau^{t}\GG= \Supp\,\GG.$
Therefore, $(F,G,\tau^{-t}P_j)$ is a $s.s$ if, and only if, $j \in \{0,p+q-1 \}$.

\noindent {\bf Case 3.} If $t$ is such that $q|t$ and   $t \equiv r\,\,(\mod\,p)$, with 
 $1 \leq r \leq p-1$, then by \ref{suporte} we have that 
$(\Supp\,\tau^{t}\FF)' =\{p-r\}  \text{ and }$
$(\Supp\,\tau^{t}\GG)' =\{0, \ldots, p-1\}\cup \{p+q-1\}.$
Therefore, in this case,  $(F,G,\tau^{-t}P_j)$ is a $s.s.$ if, and only if, $j=p-r$.

\noindent{\bf Case 4.} If $t$ is such that $p|t$ and  $q\not| t$, then  similarly to the previous case we conclude that
$(F,G,\tau^{-t}P_j)$ is a $s.s.$ if, and only if, $j=p+q-r-1$, where $r$ is such that 
 $t \equiv r \,\,(\mod \,q)$ and $1 \leq r \leq q-1$.

\noindent{\bf Case 5.} If $t \equiv r_1\,\,(\mod \,p)$ and  $t\equiv r_2\,\,(\mod\,q)$, with $0< r_1 <p$ and 
$0<r_2<q$  then $(\Supp\, \tau^{r_1}\FF)' \cap (\Supp \tau^{r_2} \GG)'=\{p-r_1\} \cap \{p+q-r_2-1\}= \phi.$
Therefore there is no a $s.s.$  with these conditions for $t$.
\end{proof}
\begin{prop}\label{preinjetivo}If $Y$ is a preinjective $A$-module such that $(F,G,Y)$ is a $s.s.$, then
$Y$ is one of the modules in the following list
\begin{enumerate}
\item $\tau^{t}I_{p}$, with $t\geq 1$ such that   $t\equiv p-1\,\, (\mod \,p)$ and  $q|t$.
\item $\tau^{t}I_1$, with $t\geq 1$ such that  $t\equiv q-1\,\, (\mod \,q)$  and $p|t$. 
\item $\tau^{t}I_{0}$, with  $t\geq 1$ such that $t\equiv p-1 \,\,(\mod \,p)$  and  $t\equiv q-1 \,\,(\mod \,q)$.
\item $\tau^{t}I_{p+q-1}$, with  $t\geq 1$ such that $t\equiv p-1 \,\,(\mod \,p)$  and  $t\equiv q-1 \,\,(\mod \,q)$.
\item $\tau^{t}I_{r+1}$, with $t\geq 1$ such that $t\equiv r\,\, (\mod \,p)$, $r\neq p-1$  and  $t\equiv q-1\,\,(\mod \,q)$.
\item $\tau^{t}I_{p+r}$, with $t\geq 1$ such that $t\equiv r\,\,(\mod \,q)$, $r\neq q-1$  and  $t\equiv p-1\,\,(\mod \,p)$.

\end{enumerate}

\end{prop}
\begin{proof} Let $Y$ be a  preinjective $A$-module and let $M$ be a regular $A$-module. Then $Y\cong \tau^{k}I_{j}$, 
for $k\geq 0$ and $j \in \{0,1,\ldots,p,\ldots, p+q-1\}$. Hence $\Hom_A(Y,M)=0$. On the other hand, by the  Auslander formula,
we have that $\Ext^{1}_{A}(\tau^{k}I_j, M)\cong D\Hom_A(\tau^{-1}M,\tau^{k}I_j)\cong D\Hom_A(\tau^{-(k+1)}M,I_j).$
It follows that
$\Ext^{1}_{A}(\tau^{k}I_j, M)=0  \text{ if, and only if, } j \notin \supp\,\tau^{-(k+1)}M.$
Therefore we conclude that 
\[(F,G,\tau^{k}I_j) \text{ is a $s.s.$ } \text{ if, and only if, } j \in (\Supp \tau^{-(k+1)}\FF)'\cap (\Supp \tau^{-(k+1)}\GG)'.\]
We consider several cases in order to find all the $s.s.$  of the  form  $(F,G,Y)$ with $Y$ preinjective.

\noindent {\bf Case 1.} Let $Y \cong I_j$ with  $j\in \{0,1,\ldots,p,\ldots, p+q-1\}$.
Then,  by \ref{suporte}, we  have that $(\Supp \tau^{-1}\FF)' \cap (\Supp \tau^{-1}\GG )' = \{1\}\cap\{p\}= \phi.$
Therefore there is no $j$ such that $(F,G,I_j)$ is a $s.s.$.

\noindent {\bf Case 2.} Let $Y \cong \tau^{t}I_j$, with $t\geq 1$ such that  $p|t$ and $q|t$. Then
$$(\Supp\tau^{-(t+1)}\FF)' \cap  (\Supp\tau^{-(t+1)}\GG)'=(\Supp \tau^{-1}\FF)'\cap (\Supp \tau^{-1}\GG)' =\phi.$$

\noindent {\bf Case 3.} Let $Y \cong \tau^{t}I_j$, with $t$ such that $p \not|t$ and  $q|t$. Let
$t \equiv r\,\,(\mod\, p)$, $1 \leq r \leq p-1$. Here we need to consider two cases: when $r=p-1$ and
when $r\not=p-1$.
\begin{itemize}
\item If $r=p-1$, then using \ref{suporte} we have that
$$ \Supp \tau^{-(t+1)}\FF =\Supp \tau^{-(r+1)}\FF = \Supp \tau^{-p}\FF = \Supp \FF \text{ and }$$ 
$$\Supp \tau^{-(t+1)}\GG = \Supp \tau^{-1}\GG=\{0, \ldots, p+q-1\} \setminus\{p\}.$$
Hence $(\Supp \tau^{-(t+1)}\FF)' \cap \Supp (\tau^{-(t+1)}\GG)' =\{p\}$ and 
therefore  $(F,G,\tau^t I_{p})$ is a $s.s.$.
\item If $r \not =p-1$, we have that
$$\Supp\tau^{-(t+1)}\FF =\Supp\tau^{-(r+1)}\FF =\{0,\ldots,p+q-1\}\setminus\{r+1\}.$$
Then $(\Supp\tau^{-(t+1)}\FF)'\cap(\Supp\tau^{-(t+1)}\GG)' =\{r+1\}\cap \{p\}=\phi $. Hence 
there is no a $j$ such that $(F,G,\tau^{t}I_j)$ is a $s.s.$
\end{itemize}

\noindent {\bf Case 4.} Let $Y \cong \tau^{t}I_j$ with $t\geq 1$ such that $q \not|t$  and $p|t$. Let
$r\equiv t\,\,(\mod q)$, with $1\leq r\leq q-1$. Here we need to consider two cases: when  $r=q-1$ and 
when $r\not= q-1$.
\begin{itemize}
\item If $r=q-1$, then
$$\Supp \tau^{-(t+1)}\GG =\Supp \tau^{-(r+1)}\GG =\Supp \tau^{-q}\GG= \Supp \GG. \text{ and }$$
$$\Supp \tau^{-(t+1)}\FF = \Supp \tau^{-1}\FF = \{ 0, \ldots, p+q-1\} \setminus \{1\}.$$
It follows that $(\Supp  \tau^{-(t+1)}\FF)' \cap (\Supp \tau^{-(t+1)}\GG)' =\{1\}.$ Therefore $(F,G,\tau^t I_{1})$
is a $s.s.$
\item If $r \not =q-1$, we have that 
$$\Supp \tau^{-(t+1)}\GG = \Supp \tau^{-(r+1)}\GG  = \{0, \ldots p+q-1\} \setminus \{p+r\}. $$
Then $(\Supp \tau^{-(t+1)}\FF)'\cap (\Supp \tau^{-(t+1)}\GG)' = \{1\}\cap \{p+r\}=\phi$ and therefore 
there is no a $j$ such that  $(F,G,\tau^{t}I_j)$  is a $s.s.$
\end{itemize}

\noindent{\bf Case 5.} Let $Y \cong \tau^{t}I_j$, with $t$ such that $q \not|t$ and $p \not|t$.
Let $r_1$ and  $r_2$ such that $t \equiv r_1 \,\,(\mod\,p)$, $t \equiv r_2 \,\,(\mod\,q)$,
$1 \leq r_1 \leq p-1$ and $1 \leq r_2 \leq q-1$. We consider several cases:
\begin{itemize}
\item Let $r_1=p-1$ and let $r_2 = q-1$. We have that 
$$(\Supp \tau^{-(t+1)}\FF)'\cap (\Supp \tau^{-(t+1)}\GG)' =(\Supp \FF)' \cap (\Supp \GG)' =\{ 0, p+q-1\}.$$
Therefore $(F, G, \tau^{t}I_0)$ and $(F, G,\tau^{t}I_{p+q-1})$ are $s.s.$.
\item Let $r_2=q-1$ and let $r_1 \not = p-1$. Then
$$\Supp \tau^{-(t+1)}\GG = \Supp \GG  \text{ and  }$$                
$$\Supp \tau^{-(t+1)}\FF = \Supp \tau^{-(r_{1}+1)}\FF =\{0,\ldots, p+q-1\} \setminus \{r_{1}+1\}. $$
Hence $(F, G,\tau^t I_{r_{1}+1})$ is a $s.s.$.
\item Let $r_1=p-1$ and let $r_2 \not = q-1$. In this case we have that 
$$\Supp \tau^{-(t+1)}\FF =\Supp \FF \text{ and  }$$
$$\Supp \tau^{-(t+1)}\GG = \Supp \tau^{-(r_{2} +1)}\GG  =\{0,\ldots p+q-1\} \setminus \{p+r_{2}\}.$$
Then $(F, G,\tau^t I_{p+r_{2}})$ is a $s.s.$.
\item If $r_1 \not=p-1$ and $r_2 \not= q-1$, then
$$\Supp \tau^{-(t+1)}\GG = \Supp \tau^{-(r_{2} +1)}\GG =\{0,\ldots p+q-1\} \setminus \{p+r_{2}\} \text{ e }$$
$$\Supp \tau^{-(t+1)}\FF = \Supp \tau^{-(r_{1} +1)}\FF =\{0,\ldots, p+q-1\} \setminus \{r_{1}+1\}. $$
Consequently, there is no a $s.s.$ in this case.
\end{itemize}
\end{proof}

The following result is well known and it is not difficult to prove.

\begin{lemma}\label{moe} Given a finite dimensional  $K$-algebra $B$ and
\begin{displaymath}
\xymatrix{
0 \ar[r] & L \ar[r]^{(f' \,\, g')^{t}} & X \oplus Y \ar[r]^{(f\,\,g)}&  M \ar[r]& 0 
}
\end{displaymath}
an almost split sequence in $\mod B$ . Then
\begin{enumerate}
\item If $f'$ is a monomorphism (resp. epimorphism), then $g$ is a monomorphism (resp. epimorphism).
\item If $g'$ is a monomorphism (resp. epimorphism), then $f$ is a monomorphism (resp. epimorphism).
\end{enumerate}
\end{lemma}
\begin{prop}\label{fcpp} Let $A=K\Delta$. The postprojective component $\PP(A)$ of $\Gamma(\mod A)$ has 
the following properties:
\begin{enumerate}
\item  Any  irreducible morphism $W \lfd V$ between indecomposable modules in  $\PP(A)$ is a monomorphism.
\item The module $P_{p+q-1}$ and all their sucessors in $\PP(A)$ are sincere.
\item The integer $p$ is the minimal with the property that the modules $\tau^{-r}P_i$ with $i\in \{0,\ldots,p+q-1\}$ 
      are sincere.
\item If $0 \leq i \leq p-1$, then
\begin{itemize}
      \item The smallest integer $r$ such that $\tau^{-r}P_i$ is sincere is $r=p-i$. Furthermore, all
       the $A$-modules of the form $\tau^{-k}P_i$, with  $k>p-i$, are sincere.
      \item If $0<k<p-i$, then the composition factors of  $\tau^{-k}P_i$ are
      $S_j$ with $0 \leq j \leq i+k$ and $p \leq j \leq p+k-1$.
      \item The composition factors  of $P_i$ are $S_j$ with $0\leq j\leq i$.
      \end{itemize}
\item If $p \leq i \leq p+q-2$, then
      \begin{itemize}
      \item If $i\leq q-1$, then the smallest integer $r$ such that $\tau^{-r}P_i$ is sincere is
       $r=p.$ Furthermore, all the $A$-modules in the form $\tau^{-k}P_{i}$, with $k>p$, are sincere.
      \item If $q-1\leq i\leq p+q-1$, then the smallest integer $r$ such that $\tau^{-r}P_i$ is sincere 
      is $r=p+q-1-i.$
      Furthermore, all the $A$-modules of the form $\tau^{-k}P_{i}$, with $k>p+q-1-i$, are sincere.
      \item If $\tau^{-k}P_{i}$ is not sincere, then their composition factors are $S_j$ with $p\leq j \leq i+k$ 
      and $0\leq j\leq k$.
      \end{itemize} 
     \item If $1\leq k \leq p-1$, then the composition factors of $\tau^{-k}P_{0}$ are $S_j$ with $0\leq j \leq k$ 
     and $p\leq j \leq p-1+k$.

\end{enumerate}

\end{prop}
\begin{proof} Our proof use the structure of the postprojective component $\PP(A)$ of $\Gamma(\mod A)$ which looks as follows
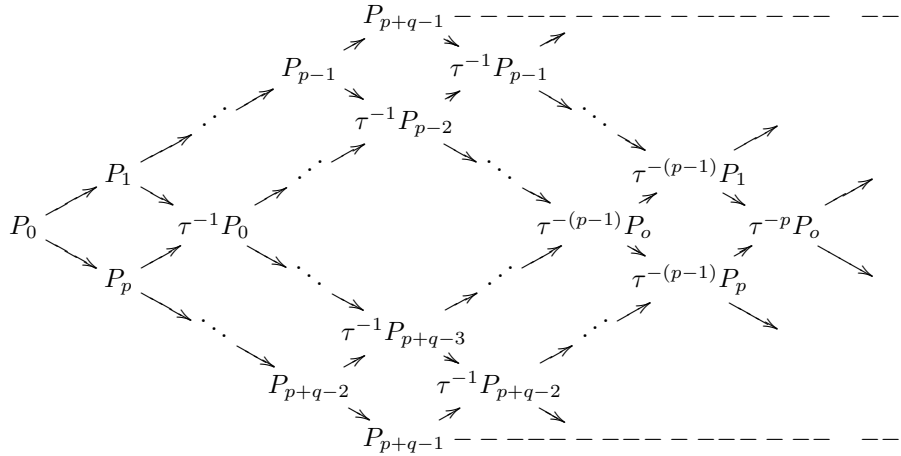
\begin{displaymath}
\xymatrix @!0 @R=2em @C=3pc{
& & & & P_{p+q-1} \ar[dr]&----&---- &----&----&--\\
& & & P_{p-1}\ar[dr]\ar[ur] & & \tau^{-1}P_{p-1}\ar[dr]\ar[ur] & & &&\\
& & \iddots \ar[ur] & & \tau^{-1}P_{p-2}\ar[ur]\ar[dr] & & \ddots \ar[dr] & &&\\
& P_{1} \ar[ur]\ar[dr] & & \iddots  \ar[ur] && \ddots \ar[dr] & & \tau^{-(p-1)}P_{1}\ar[dr] \ar[ur] &&\\
P_0 \ar[ur]\ar[dr] & &\tau^{-1}P_{0}\ar[ur] \ar[dr]& & & & \tau^{-(p-1)}P_{o}\ar[dr]\ar[ur] & & \tau^{-p}P_{o} \ar[ur] \ar[dr]&\\
& P_{p}\ar[ur]\ar[dr] & & \ddots \ar[dr] & & \iddots \ar[ur]&& \tau^{-(p-1)}P_p \ar[ur]\ar[dr]&&\\
& &\ddots\ar[dr] & & \tau^{-1}P_{p+q-3}\ar[ur] \ar[dr] & & \iddots \ar[ur] &&&\\
& & &P_{p+q-2}\ar[ur]\ar[dr]& & \tau^{-1}P_{p+q-2}\ar[ur]\ar[dr]&&&&\\
& & & &P_{p+q-1}\ar[ur]&----&----&----&----&--\\
}
\end{displaymath}

In order to prove (1) we observe that in the almost split sequence
$$ 0 \lfd P_0 \overset{(\alpha'\,\,\,\beta')^{t}} \lfd P_1\oplus P_p \overset{(\alpha\,\,\,\beta)}\lfd \tau^{-1}P_0 \lfd 0,$$
the morphisms $\alpha'$ and $\beta'$ are mono, because are irreducible morphisms betwen projective modules. Then, by \ref{moe}, 
$\alpha$ and  $\beta$ are mono too. Analogously, in all almost split sequences begining in a projective module the arrows represent 
monomorphisms. Hence, according to the shape $\PP(A)$, all the arrows in this component represent monomorphisms.

Now we prove (2). We note that in $\PP(A)$ there are paths 
$$P_0\fd P_1\fd\cdots\fd P_{p-1}\fd P_{p+q-1} \text{ and  }P_0 \lfd P_p \fd \cdots\fd P_{p+q-1} \fd P_{p+q-1},$$
where, by (1), all the arrows represent monomorphisms. Therefore, $P_{p+q-1}$ and all their sucessors are  sincere modules.

For (3), we observe that none of the predecessors  in  $\PP(A)$ of modules in the paths  below 
\begin{equation}\label{caminho1}
P_{p+q-1}\lfd \tau^{-1}P_{p-1} \lfd \tau^{-2}P_{p-2}\lfd \cdots \lfd \tau^{-(p-1)}P_{0} \lfd  \tau^{-p}P_{0} \end{equation}
\begin{equation}\label{caminho2}
P_{p+q-1}\lfd \tau^{-1}P_{p+q-2} \lfd \tau^{-2}P_{p+q-3}\lfd \cdots \lfd \tau^{-(p-1)}P_{p} \lfd \tau^{-p}P_{0} \end{equation}
have  $S_{p+q-1}$ as composition factor. In particular, $\tau^{-(p-1)}P_{0}$ don't have $S_{p+q-1}$ as composition factor. 
Therefote, by (1), none of their predecessors have $S_{p+q-1}$ as composition factor. On the other hand, any successor of 
$\tau^{-(p-1)}P_{0}$, and any of modules in (\ref{caminho1}) and (\ref{caminho2}), are sincere.

In order to show (4), we note that all the modules  $\tau^{-(p-i)}P_{i}$, with $0 \leq i \leq p-1 $,
are in the path (\ref{caminho1}) and therefore are sincere. Moreover, as stated earlier, any  predecessor of some of the
 modules which appear in (\ref{caminho1}) have no $S_{p+q-1}$ as composition factor. It follows that, if 
$0 \leq i \leq p-1$ then none of the modules of the form $\tau^{-k}P_{i}$, with $k<p-i$, are sincere. 

On the other hand, the predeccessors of  $P_i$, with  $0\leq i \leq p-1$, are the projectives
$P_0, P_1,\ldots,P_{i-1}$. Thus the compostion factors of $P_i$,  with $0\leq i \leq p-1$,
are $S_j$ with $0\leq j \leq i.$ Analogously, for $p \leq i <p+q-1$, the compostion factors of
$P_i$ are $S_j$, with $ p \leq j \leq i$ and $j=0$. 

Let $0 \leq i \leq p-1 $. Then in  $\PP(A)$ there are paths
\begin{equation}
P_{i+k}\fd \tau^{-1}P_{i+k-1} \fd \tau^{-2}P_{i+k-2}\fd \cdots \fd \tau^{-(k-1)}P_{i+1} \fd  \tau^{-k}P_{i} \end{equation}
\begin{equation}
\begin{array}{l}
 P_{p+k-1}\fd \tau^{-1}P_{p+k-2} \fd \tau^{-2}P_{p+k-3}\fd \cdots \fd \tau^{-(k-1)}P_{p} \fd \tau^{-k}P_{0} \fd \\
  \tau^{-k}P_{1}\fd  \tau^{-k}P_{2}\fd            \cdots \fd \tau^{-k}P_{i}
\end{array}
\end{equation}

Furthermore,  there are no paths between $P_{j}$, with $j \geq i+k$, and  $\tau^{-k}P_{i}$. Hence the composition factors of 
$\tau^{-k}P_i$, up to multiplicity, are the composition factors of $P_{i+k}$ and of $P_{p+k-1}.$ In other words, the composition 
factors of  $\tau^{-k}P_{i}$ are  $S_j$ with $0\leq j \leq i$ and $p \leq i <p+q-1$.

The proof of (5), is analogous to the previous  and part (6) is left to the reader, since it follows a similar argument.
\end{proof}

\begin{prop}\label{completo1} Let $A=K\Delta$. The  complete list of $s.s.$
$(X,F,G,Y)$, where $Y$ is a postprojective $A$-module, is as follows:
\begin{enumerate}
\item $(S_{p+q-1},F,G, P_0)$.
\item $(X,F,G, P_{p+q-1})$, where
      \begin{displaymath}
      \begin {array}{rll}
      X \cong &\left \{ \begin {array}{ll}
      \tau^{-p+1}P_{0}, \textrm{ if } p=q\\
      \tau^{-p+1}P_{q-1}, \textrm{ if } p \not =q\\
      \end{array} \right.\\
      \end{array}
      \end{displaymath}
\item $(X,F,G,\tau^{-t} P_{p+q-1})$ with $t\in \mathbb{N}$ such that $p|t$ and $q|t$, where
\begin{displaymath}
      \begin {array}{rll}
      X \cong &\left \{ \begin {array}{ll}
      \tau^{-t-p+1}P_{0}, \textrm{ if } p=q\\
      \tau^{-t-p+1}P_{q-1}, \textrm{ if } p \not =q\\
      \end{array} \right.\\
      \end{array}
\end{displaymath}
\item $(\tau^{-t+1}P_{p+q-1},F,G,\tau^{-t}P_0)$, with $t\in\mathbb{N}$
      such that $p|t$ and $q|t$.
\item $(\tau^{-t-(p-r-1)}P_{q+r-1},F,G,\tau^{-t} P_{p-r})$, with $t,r\in \mathbb{N}$ such that $q|t$, $t\equiv r \,\,(\mod \,p)$ 
      and $ 1 \leq r \leq p-1$. 
\item $(X, F,G,\tau^{-t}P_{p+q-r-1})$, with $t,r\in\mathbb{N}$ such that $p|t$, $t\equiv r \,\,(\mod \,q)$
      and $ 1 \leq r \leq q-1$, where
      \begin{displaymath}
      \begin {array}{rll}
      X \cong &\left \{ \begin {array}{ll}
      \tau^{-t-q+r+1}P_{p-q+r}, \textrm{ if } p \leq q-r\\
      \tau^{-t-p+1}P_{q-r-1}, \textrm{ if } p> q-r\\
      \end{array} \right.\\
      \end{array}
\end{displaymath}
\end{enumerate}
\end{prop}
\begin{proof} The Proposition \ref{posprojetivo1}  characterized $s.s.$ of the form $(F, G, Y)$,  with $Y$
a  postprojective $A$-module. By using this characterization we obtain stratifying systems of type $(X,F,G,Y)$.
 
The quiver $\Delta$ has $p+q$ vértices then, by \ref{estender} and \ref{unico},
the module $X$ such that  $(X,F,G,\tau^{-t}P_l)$ is a $c.s.s.$ is unique. Thus, the proof consist in to verify 
that the indecomposable module $ X $ satisfies the following conditions:
\begin{itemize}
\item $\Ext_{A}^{1}(X,X)=0$.
\item  $\Hom_{A}(\FF,X)=0$ and $\Hom_{A}(\GG,X)=0$.
\item  $\Ext_{A}^{1}(\FF,X)=0$ and $\Ext_{A}^{1}(\GG,X)=0$.
\item  $\Hom_{A}(\tau^{-t}P_l,X)=0$ and $\Ext_{A}^{1}(\tau^{-t}P_l,X)=0$.
\end{itemize}
We observe that all the modules $X$ considered in this proof, except $I_ {p+q-1}$,
are of the form $X\cong\tau^{-k}P_j$, with $k \geq 0$, and therefore are indecomposable and have no self-extensions.
Moreover, there is no nonzero morphism from a regular module to a posprojective module, that is, if $R$ is a regular 
module then $\Hom_A(R,\tau^{-k}P_j)=0$. On the other hand, by the Auslander formula, we have that 
$$\Ext^{1}_A(R, \tau^{-k}P_{j}) \cong  D\Hom_A(\tau^{-k}P_{j}, \tau R)\cong D\Hom_A(P_{j}, \tau^{k+1}R).$$
Therefore we have that $\Ext^{1}_A(\FF,\tau^{-k}P_{j})=0 \text{ and }\Ext^{1}_A(\GG,\tau^{-k}P_{j})=0$ if, and only if, 
$ j\in (\Supp \tau^{k+1}\FF)' \cap (\Supp \tau^{k+1} \GG)'$.
In view of these observations, if $(F,G,\tau^{-t}P_l)$ is a $s.s.$ and
$X\cong\tau^{-k}P_j$, in order to show that $(X, F,G,\tau^{-t}P_l)$
is a $s.s.$ is sufficient to check the following conditions:                                                                            
\begin{itemize}
\item $j\in (\Supp \tau^{k+1}\FF)' \cap (\Supp \tau^{k+1} \GG)'.$
\item $\Hom_{A}(\tau^{-t}P_l,X)=0$  (ou equivalently $[\tau^{t}X:S_l]=0$).
\item  $\Ext_{A}^{1}(\tau^{-t}P_l,X)=0$.
\end{itemize}
In what follows we find these conditions for each of the sequences stated in the proposition.
\noindent (1) First let us check that $(S_{p+q-1},F,G, P_0)$  is a $s.s.$. As the vertice $p+q-1$ is a source then 
the  simple $A$-module $S_{p+q-1}$ is injective and thus 
      $$\Ext_{A}^{1}(\FF,S_{p+q-1})=0,\Ext_{A}^{1}(\GG,S_{p+q-1})=0 \text{ and }\Ext_{A}^{1}(P_{0},S_{p+q-1})=0.$$
On the other hand,  since  $p+q-1\notin(\Supp\,\FF \cup\Supp\,\GG)$ and  $S_{p+q-1}\cong I_{p+q-1}$  then 
    $$\Hom_{A}(\FF,S_{p+q-1})=0, \Hom_{A}(\GG, S_{p+q-1})=0 \text{ and  }\Hom_{A}(P_0, S_{p+q-1})=0.$$
    
\noindent (2) To complete the $s.s.$ $(F,G, P_{p+q-1})$  we have to consider two cases:
      \begin{itemize}
      \item If $p=q$ and $X \cong \tau^{-p+1}P_0$, then  by \ref{suporte} we have that
      $(\Supp \tau^{p}\FF)' \cap (\Supp \tau^{p} \GG)'= \{0, p+q-1\}.$
      On the other hand $[\tau^{-p+1}P_{0}:S_{p+q-1}]=0$, by \ref{fcpp} (6).
       Therefore the sequence $( \tau^{-p+1}P_0 ,F,G, P_{p+q-1})$
      is a $s.s.$.
    
            \item If $p\not=q$ and $X= \tau^{-p+1}P_{q-1}$ then we have that $(\Supp\tau^{p}\FF)' \cap (\Supp\tau^{p}\GG)'= \{q-1\}$
            and, by \ref{fcpp} (5), we have that $[\tau^{-p+1}P_{q-1}:S_{p+q-1}]=0$.
            \end{itemize}
In both cases we have the item 2 of the list.
 \vspace{0.3cm}

\noindent (3) In order to complete the $s.s.$ $(F,G,\tau^{-t} P_{p+q-1})$, with $t\geq 1$ such that
      $p|t$ and  $q|t$, we consider the following two cases:
\begin{itemize}
\item Suppose $p=q$. We claim that the sequence $(\tau^{-t-p+1}P_{0},F, G,\tau^{-t}P_{p+q-1})$, is a $s.s.$. In fact,
      $ (\Supp\tau^{t+p}\FF)' \cap (\Supp\tau^{t+p}\GG)'=\{0, p+q-1\}.$
     
      On the other hand, by \ref{fcpp} (6), 
      $\Hom_A(\tau^{-t}P_{p+q-1},\tau^{-t-p+1}P_{0})\cong 0.$
      Moreover, by \ref{repete} (1), \\$\Ext^{1}_A(\tau^{-t}P_{0},\tau^{-t-p+1}P_{q-1})=0.$
     These conditions complete the verification of our claim.
      \item Suppose $p\not =q$. In this case we see that $(\tau^{-t-p+1}P_{q-1},F,G,\tau^{-t}P_{p+q-1})$
      is a $c.s.s.$ Indeed, $(\Supp\tau^{t+p}\FF)'\cap (\Supp\tau^{t+p}\GG)'=\{q-1\}$ by \ref{suporte}, 
      $[\tau^{-p+1}P_{q-1}:S_{p+q-1}]=0$ by \ref{fcpp} (5)
      and finally by  \ref{repete} (1) $\Ext^{1}_A(\tau^{-t}P_{p+q-1},\tau^{-t-p+1}P_{q-1})=0$.

      \end{itemize} 
       \vspace{0.3cm}
      
\noindent (4) We show that $(\tau^{-t+1}P_{p+q-1},F, G, \tau^{-t}P_{0})$, with $t\geq 1$ such that
      $p|t$ and  $q|t$ is a $c.s.s.$. In fact,
      $(\Supp \tau^{t}\FF)' \cap (\Supp \tau^{t} \GG)'= \{ 0, p+q-1\}.$
       On the other hand, by using the Auslander formula we have 
      $$\Hom_A(\tau^{-t}P_{0},\tau^{-t+1}P_{p+q-1})\cong D\Ext^{1}_{A}(P_{p+q-1},\tau^{-1}P_{0})\cong 0$$
     
      and 
     $$ \Ext_A^{1}(\tau^{-t}P_{0},\tau^{-t+1}P_{p+q-1}) \cong D\Hom_A(P_{p+q-1},P_{0}) \cong 0. $$
     
     Then the conditions for the $s.s.$ are verified.
      \vspace{0.3cm}
      
\noindent (5) We prove that $(\tau^{-t-(p-r-1)}P_{q+r-1},F,G,\tau^{-t} P_{p-r})$, with $t,r\in \mathbb{N}$ such that  $q|t$,          
      $t\equiv r \,\,(\mod \,p)$ and $1\leq r\leq p-1$, is a $s.s.$. According to \ref{suporte}, we have 
      $$(\Supp \tau^{t+(p-r)}\FF)' \cap (\Supp \tau^{t+(p-r)} \GG)'=(\Supp \FF)' \cap (\Supp \tau^{p-r} \GG)'
                                                                 =\{q+r-1\}.$$                     
     
      Applying \ref{repete} (1) we can  assert that 
      $\Ext_A^{1}(\tau^{-t}P_{p-r},\tau^{-t-(p-r-1)}P_{q+r-1})=0.$ And  finally, by \ref{fcpp} (7), we have 
    $$\Hom_A(\tau^{-t}P_{p-r},\tau^{-t-(p-r-1)}P_{q+r-1})\cong \Hom_A(P_{p-r},\tau^{-(p-r-1)}P_{q+r-1})\cong 0.$$

     \vspace{0.3cm}
    
\noindent (6) Suppose that $t,r\in\mathbb{N}$ are such that $p|t$, $t\equiv r \,\,(\mod \,q)$
      and $ 1 \leq r \leq q-1$. Let us consider two cases:
     \begin{itemize}
     \item Assume that $p \leq q-r$. Then, by \ref{suporte}, we have that
    \[ \begin{array}{rcl}
      (\Supp\tau^{t+q-r}\FF)' \cap (\Supp \tau^{t+q-r}\GG)'& =& (\Supp \tau^{q-r} \FF)'\cap(\Supp\GG)'\\
                                                             & = &\{p-q+r \}.\\
      \end{array}\]
     We note that  $q-r-1 \geq 0 $, then  by \ref{repete} (1), have 
      $$\Ext_A^{1}(\tau^{-t}P_{p+q-r-1},\tau^{-t-q+r+1}P_{p-q+r})=0.$$
     Furthermore, by \ref{fcpp} (4) it follows that
     $$\Hom_A(\tau^{-t} P_{p+q-r-1},\tau^{-t-q+r+1} P_{p-q+r})\cong\Hom_A(P_{p+q-r-1},\tau^{-q+r+1}P_{p-q+r})\cong 0.$$
     
     Thus $(\tau^{-t-q+r+1} P_{p-q+r}, F,G, \tau^{-t} P_{p+q-r-1})$ is a $c.s.s.$.
     
     \item Suppose that  $q-r>p$. Let $l$ such that  $p+l=q-r$. Therefore by \ref{suporte}, we have that
     \[\begin{array}{rcl}
     (\Supp \tau^{t+p-1}\FF)' \cap (\Supp \tau^{t+p-1} \GG)'&=& (\Supp \tau^{-l}\FF)' \cap (\Supp \GG)' \\
                                                            &=&\{p+l-1\}\\
                                                            &=&\{q-r-1\}.\\
     \end{array}\]
     Now by \ref{repete} (1), we have that $\Ext_A^{1}(\tau^{-t}P_{p+q-r-1},\tau^{-t-p-1}P_{q-r-1})=0.$
     Finally by \ref{fcpp} (4) we have that $[S_{p+q-r-1}:\tau^{-p+1}P_{q-r-1}]=0.$
    
     We conclude that $(\tau^{-t-p+1} P_{q-r-1}, F,G, \tau^{-t} P_{p+q-r-1})$ is a $c.s.s.$.

     \end{itemize}
\end{proof}

Next result is proved in an analogous way as \ref{fcpp}.  

\begin{prop}Let $A=K\Delta$. The preinjective component $\QQ(A)$ of $\Gamma(\mod A)$ has the following properties:
\begin{enumerate}
\item Any irreducible  morphism $W\lfd V$ between indecomposable modules in $\QQ(A)$ is an epimorphism.
\item The $A$-module $I_{0}$  and all their predecessors in $\QQ(A)$ are  sincere modules.
\item The integer $p$ is minimal with the property that the modules  $\tau^{r}I_i$, with $i\in \{0,\ldots, p+q-1\}$
are sincere.
\item If $1\leq i \leq p-1$, then:
      \begin{itemize}
      \item the smallest integer $r$ such that $\tau^{r}I_{i}$ is a sincere module is $r=i$. 
      Furthermore, all the $A$-modules in the form $\tau^{k}I_{i}$, with $k\geq i$, are  sincere $A$-modules. 
      \item if $k<i$, then the composition factors of $\tau^{k}I_{i}$ are simple $A$-modules  $S_j$ with $i-k\leq j \leq p-1$ and
       $p+q-k \leq j \leq p+q-1$.      
      \end{itemize}  
\item If $p\leq i< p+q-1$, then:
      \begin{itemize} 
      \item if  $p\leq i< q-1$, the smallest integer $r$ such that $\tau^{r}I_{i}$ is a sincere module is 
      $r=i-p+1.$  Moreover,  all the $A$-modules of the form $\tau^{r}I_{i}$, with $r \geq i-p+1$,
      are sincere.           
      \item if $q-1\leq i\leq p+q-1$, the smallest integer $r$ such that $\tau^{r}I_{i}$ is a sincere $A$-module is  $r=p$. 
      Furthermore, all the $A$-modules of the form $\tau^{r}I_{i}$, with $r\geq p$ are sincere.
      \item if $p\leq i\leq p+q-2$ and  $k$ is such that  $\tau^{k}I_{i}$ is not a sincere module, then the composition
      factors of  $\tau^{k}I_{i}$ are simple $A$-modules  $S_j$ with $i-k\leq j\leq p+q-1$
      and $p-k \leq j\leq p-1$.
      \end{itemize}
\item If $k<p$, then  the composition of $\tau^{k}I_{p+q-1}$ are simple $A$-modules  $S_j$ with $p-k\leq j\leq p-1$
      and $p+q-k \leq j \leq p+q-1$.
     
\end{enumerate}
\end{prop}
\begin{prop}\label{completo2} \label{fcpi}Let $A=K\Delta$. The complete list of $s.s.$
$(X,F,G,Y)$, where $Y$ is a preinjective $A$-module is as follows:
\begin{enumerate}

\item $(\tau^{t}I_{p-1},F,G,\tau^{t}I_{p})$, with $t\geq 1$ such that $t\equiv p-1\,\,(\mod\,p)$ and $q|t$.
\item $(\tau^{t}I_{p+q-2},F,G,\tau^{t}I_1)$, with $t\geq 1$ such that $t\equiv q-1\,\,(\mod\,q)$ and $p|t$. 
\item $(\tau^{t+1}I_{p+q-1},F,G,\tau^{t}I_{0})$, with $t\geq 1$ such that $t\equiv p-1\,\,(\mod\,p)$ and  $t\equiv q-1\,\,(\mod\,q)$.
\item $(\tau^{t-p+1}I_{q-1},F,G,\tau^{t}I_{p+q-1})$, with $t\geq 1$ such that $t\equiv p-1\,\,(\mod\,p)$ and $t\equiv q-1 \,\,(\mod \,q)$.
\item $(\tau^{t-r} I_{p+q-r-2},F,G,\tau^{t}I_{r+1})$, with $t\geq 1$ such that $t\equiv r\,\,(\mod\,p)$,
      $0<r<p-1$ e $t\equiv q-1\,\,(\mod\,q)$.
\item $(X,F,G,\tau^{t}I_{p+r})$, with $t\geq 1$ such that $t\equiv r\,\,(\mod\,q)$, $0<r<q-1$ and $t\equiv p-1\,\,(\mod\,p)$, where 
      \begin{displaymath}
      \begin {array}{rll}
      X \cong &\left \{ \begin {array}{ll}
      \tau^{t-r}I_{p-(r+1)}, \textrm{ if } r < p\\
      \tau^{t-(p-1)}I_{r}, \textrm{ if } r \geq p\\
      \end{array} \right.\\
      \end{array}
      \end{displaymath}

\end{enumerate}
\end{prop}

\begin{proof} By \ref{estender} and  \ref{unico} there is an unique $A$-module $X$ such that $(X,F,G,\tau^{m}I_i)$ 
is a $s.s.$. Therefore the proof consists in  verify 
that the indecomposable module $ X $ satisfies the following conditions: 
\begin{itemize}
\item $\Ext_{A}^{1}(X,X)=0$.
\item  $\Hom_{A}(\FF,X)=0  \text{ and }\Hom_{A}(\GG,X)=0$.
\item  $\Ext_{A}^{1}(\FF,X)=0 \text{ and  } \Ext_{A}^{1}(\GG,X)=0$.
\item  $\Hom_{A}(\tau^{t}I_i,X)=0 \text{ and }\Ext_{A}^{1}(\tau^{t}I_i,X)=0$.
\end{itemize}
All modules $X$ that we will consider are of the form $X\cong\tau^{l}I_j$ and therefore are
indecomposable and have no self-extensions. Moreover if  $X$ is a preinjective $A$-module and $R$ is a regular $A$-module then, by using the Auslander formula, we have $\Ext^{1}_A(R,X)\cong D\Hom_A(X,\tau R)\cong 0$.
On the other hand, for $l\geq 0$, we have $\Hom_A(R,\tau^{l}I_j) \cong \Hom_A(\tau^{-l}R,I_j).$ It follows that
$\Hom_A(R,\tau^{l}I_j)=0 \Leftrightarrow j \notin \supp\tau^{-l}R.$ Therefore we have
$$\Hom_{A}(\FF,\tau^{l}I_j)=0 \text{ and }\Hom_{A}(\GG,\tau^{l}I_j)= 0 \Leftrightarrow j \in (\Supp\tau^{-l}\FF)'\cap (\Supp\tau^{-l}\GG)'.$$

According to the above remarks, if $(F,G,\tau^{t}I_i)$ is a $s.s.$ and $X\cong\tau^{l}I_j$, in order to prove 
that $(X,F,G,\tau^{t}I_i)$ is $c.s.s.$ is sufficient to show the following conditions:
\begin{enumerate}
\item[a. ]$j \in (\Supp \tau^{-l}\FF)'\cap (\Supp \tau^{-l}\GG)'$
\item[b. ] $\Hom_{A}(\tau^{t}I_i,X)=0$
\item[c. ] $\Ext_{A}^{1}(\tau^{t}I_i,X)=0$.
\end{enumerate}

We consider various possibilities, which dependend of the form of $t$, according to the list of the statement of \ref{preinjetivo}.

\noindent(1) Let $t\geq 1$ such that $t\equiv p-1\,\,(\mod \,p)$ and $q|t$. Then, by \ref{suporte}, we have 
      $$(\Supp \tau^{-t}\FF)' \cap(\Supp \tau^{-t}\GG)'=(\Supp \tau^{-(p-1)}\FF)'\cap(\Supp \tau^{-t}\GG)'= \{p-1\}.$$
      We observe that $\Hom_A(\tau^{t}I_{p},\tau^{t}I_{p-1})\cong \Hom_A(I_{p},I_{p-1}) \cong 0$, because $I_p$ has no $S_{p-1}$
      as composition factor.
     
      Moreover, by \ref{repete} (2), $\Ext^{1}_A(\tau^{t}I_{p},\tau^{t}I_{p-1})=0.$  Then 
      $(\tau^{t}I_{p-1},F,G,\tau^{t}I_{p})$ is a $s.s.$

\noindent(2) Let $t\geq 1$ such that $t\equiv q-1\,\,(\mod \,q)$ and $p|t$. By \ref{suporte} we have that
      $$(\Supp \tau^{-t}\FF)'\cap(\Supp \tau^{-t}\GG)'=(\Supp \tau^{-t}\FF)'\cap(\Supp \tau^{-(q-1)}\GG)'=\{p+q-2\}.$$
      On the other hand, provided that $I_1$ has no $S_{p+q-2}$ as composition factor then $\Hom_A(\tau^{t}I_{1},\tau^{t}I_{p+q-2})\cong \Hom_A(I_{1},I_{p+q-2})\cong 0$.
     
   Finally by \ref{repete} it follows that $\Ext^{1}_A(\tau^{t}I_{1},\tau^{t}I_{p+q-2})=0.$
    Therefore     $(\tau^{t}I_{p+q-2},F,G,\tau^{t}I_{1})$, with $t\geq 1$ such that $t\equiv q-1\,\,(\mod\,q)$ and $p|t$,        
     is a  $c.s.s.$.
      
\noindent(3) Let $t\geq 1$ such that $t\equiv p-1\,\,(\mod \,p)$ eand $t\equiv q-1\,\,(\mod\,q)$. By  \ref{suporte} we have that
     $$(\Supp \tau^{-t-1}\FF)'\cap(\Supp \tau^{-t-1}\GG)'= (\Supp \FF)'\cap(\Supp \GG)'=\{0, p+q-1\}.$$
    
     From the Auslander formula it may be concluded that 
     $$\Hom_A(\tau^{t}I_{0},\tau^{t+1}I_{p+q-1})\cong\Hom_A(I_{0},\tau I_{p+q-1})\cong D\Ext_A^{1}(I_{p+q-1},I_{0})=0$$
    and that
    $$\Ext^{1}_A(\tau^{t}I_{0},\tau^{t+1}I_{p+q-1})\cong D\Hom_{A}(\tau^{t+1}I_{p+q-1},\tau^{t+1}I_{0})
                                                   \cong \Hom_A(I_{p+q-1},I_{0})
                                                  =0.$$
     
    Consequently $(\tau^{t+1}I_{p+q-1},F,G,\tau^{t}I_{0})$, with $t\geq 1$ such that $t\equiv q-1\,\,(\mod\,q)$ and
     $t\equiv p-1\,\,(\mod\,p)$ is a $s.s.$
     
\noindent(4) Let  $t\geq 1$ such that  $t\equiv p-1 \,\,(\mod \,p)$ and  $t\equiv q-1 \,\,(\mod \,q)$. We show that 
      $(\tau^{t-p+1}I_{q-1},F,G,\tau^{t}I_{p+q-1})$ is a $c.s.s.$ over $A$. By
      \ref{suporte} follows that
      $$ (\Supp \tau^{-t+p-1}\FF)'\cap(\Supp \tau^{-t+p-1}\GG)'=(\Supp \FF)'\cap(\Supp \tau^{-(q-p)}\GG)'
                                                           =\{q-1\}.$$
      Thus,  $\Hom_A(\tau^{t}I_{p+q-1},\tau^{t-(p-1)}I_{q-1})\cong \Hom_A(\tau^{p-1}I_{p+q-1},I_{q-1})=0,$
      because by  \ref{fcpi} (5)  we have that  $[\tau^{p-1}I_{p+q-1}:S_{q-1}]=0.$  Moreover by  \ref{repete} (2), 
      we have that  $\Ext^{1}_A(\tau^{t}I_{p+q-1},\tau^{t-(p-1)}I_{q-1})=0.$
      
\noindent(5) Let $t\geq 1$ such that $t\equiv r\,\,(\mod\,p)$, $0<r<p-1$ and $t\equiv q-1\,\,(\mod\,q)$ and 
      $X\cong \tau^{t-r} I_{p+q-r-2}$.  First, by \ref{suporte}, we have that
      \[\begin{array}{rcl}
      (\Supp\tau^{-(t-r)}\FF)'\cap(\Supp\tau^{-(t-r)}\GG)'&=&(\Supp\FF)'\cap (\Supp\tau^{-(t-r)}\GG)'\\
                                                          &=& (\Supp\tau^{-(q-1-r)}\GG)'\\
                                                          &=&\{p+q-r-2\}.
      \end{array} \]
      Furthermore, $\Hom_A(\tau^{t}I_{r+1},\tau^{t-r}I_{p+q-r-2})\cong \Hom_A(\tau^{r}I_{r+1},I_{p+q-r-2})=0,$
      because,  by  \ref{fcpi} (5), $\tau^{r}I_{r+1}$ has no $S_{p+q-r-2}$
      as composition factor. Finally, by \ref{repete} (2),  we have that  $\Ext^{1}_A(\tau^{t}I_{r+1},\tau^{t-r}I_{p+q-r-2})=0.$ 
      Then  $(\tau^{t-r} I_{p+q-r-2},F,G,\tau^{t}I_{r+1})$ is a $c.s.s.$ over $A$.

\noindent(6) Let  $t\geq 1$ such that $t\equiv r\,\,(\mod\,q)$, $0<r<q-1$ and  $t\equiv p-1\,\,(\mod\,p)$.
      To complete the $s.s.$ $(F,G,\tau^{t}I_{p+r})$ we consider two situations:
      \begin{itemize}
      \item Suppose that  $r<p$.  Let $X\cong \tau^{t-r}I_{p-(r+1)}$. We have that 
       \[\begin{array}{rcl}
      (\Supp\tau^{-(t-r)}\FF)'\cap(\Supp\tau^{-(t-r)}\GG)'&=&(\Supp\tau^{-(p-1-r)}\FF)'\cap (\Supp \GG)'\\
                                                          &=& (\Supp\tau^{(r+1)}\FF)'\\
                                                          &=&\{p-(r+1)\}.\\
      \end{array} \]
      Since $[\tau^{r}I_{p+r}:S_{p-(r+1)}]=0$, by \ref{fcpi} (5), then 
      $$\Hom_A(\tau^{t}I_{p+r},\tau^{t-r}I_{p-(r+1)})\cong \Hom_A(\tau^{r}I_{p+r},I_{p-(r+1)})=0.$$
      By \ref{repete}, (2), we have that $\Ext^{1}_A(\tau^{t}I_{p+r},\tau^{t-r}I_{p-(r+1)})=0.$ Hence 
       $(\tau^{t-r}I_{p-(r+1)},F,G,\tau^{t}I_{p+r})$ with $t\equiv r\,\,(\mod\,q)$, $0<r<q-1$,    
      $t\equiv p-1\,\,(\mod\,p)$  and $r<p$  is a $s.s.$.
      
      \item Suppose that $r \geq p$ and let $X\cong \tau^{t-(p-1)}I_{r}$. Then we have that
      \[\begin{array}{rcl}
      \Supp\tau^{-[t-(p-1)]}\FF)'\cap(\Supp\tau^{-[t-(p-1)]}\GG)'&=&(\Supp\tau^{-[r-(p-1)]}\GG)'\\
                                                          &=& \{p+(r-p+1)-1 \}\\
                                                          &=&\{r\}. \\
      \end{array} \] 
      According to \ref{fcpi} (5), we have that  $$\Hom_A(\tau^{t}I_{p+r},\tau^{t-(p-1)}I_{r})\cong \Hom_A(\tau^{p-1}I_{p+r},I_{r}) \cong 0$$
      and by \ref{repete} we have  that $\Ext^{1}_A(\tau^{t}I_{p+r},\tau^{t-(p-1)}I_{r})=0.$ That is  the sequence $(\tau^{t-(p-1)}I_{r},F,G,\tau^{t}I_{p+r})$ with
      $t\equiv r\,(\mod\,q)$, $0<r<q-1$, $t\equiv p-1\,(\mod\,p)$ and  $r \geq p$  is a $s.s.$
     \end{itemize}
\end{proof}

Finally, \ref{completo1} and \ref{completo2} together are the main result of this section, which establish the complete list of $c.s.s.$ of the form $(X,F,G,Y)$. 

\begin{theorem} Let $A=K\Delta$. The  complete list of $s.s.$
$(X,F,G,Y)$  is as follows:
\begin{enumerate}
\item $(S_{p+q-1},F,G, P_0)$.
\item $(X,F,G, P_{p+q-1})$, where
      \begin{displaymath}
      \begin {array}{rll}
      X \cong &\left \{ \begin {array}{ll}
      \tau^{-p+1}P_{0}, \textrm{ if } p=q\\
      \tau^{-p+1}P_{q-1}, \textrm{ if } p \not =q.\\
      \end{array} \right.\\
      \end{array}
      \end{displaymath}
\item $(X,F,G,\tau^{-t} P_{p+q-1})$ with $t\geq 1$ such that $p|t$ and  $q|t$, where
\begin{displaymath}
      \begin {array}{rll}
      X \cong &\left \{ \begin {array}{ll}
      \tau^{-t-p+1}P_{0}, \textrm{ if } p=q\\
      \tau^{-t-p+1}P_{q-1}, \textrm{ if } p \not =q.\\
      \end{array} \right.\\
      \end{array}
\end{displaymath}
\item $(\tau^{-t+1}P_{p+q-1},F,G,\tau^{-t}P_0)$, with $t\geq 1$
      such that $p|t$ e $q|t$.
\item $(\tau^{-t-(p-r-1)}P_{q+r-1},F,G,\tau^{-t} P_{p-r})$, with $t\geq 1$ such that $q|t$ and  $r$ such that $t\equiv r \,(\mod \,p)$ 
      and $ 1 \leq r \leq p-1$. 
\item $(X, F,G,\tau^{-t}P_{p+q-r-1})$, with $t\geq 1$ such that $p|t$ and $r$ such that $t\equiv r \,\,(\mod \,q)$
      and $ 1 \leq r \leq q-1$.
      \begin{displaymath}
      \begin {array}{rll}
      X \cong &\left \{ \begin {array}{ll}
      \tau^{-t-q+r+1}P_{p-q+r}, \textrm{ if } p \leq q-r\\
      \tau^{-t-p+1}P_{q-r-1}, \textrm{ if } p> q-r.\\
      \end{array} \right.\\
      \end{array}
\end{displaymath}
\item $(\tau^{t}I_{p-1},F,G,\tau^{t}I_{p})$, with $t\geq 1$ such that $t\equiv p-1\,\,(\mod\,p)$ and $q|t$.
\item $(\tau^{t}I_{p+q-2},F,G,\tau^{t}I_1)$, with $t\geq 1$ such that $t\equiv q-1\,\,(\mod\,q)$ and $p|t$. 
\item $(\tau^{t+1}I_{p+q-1},F,G,\tau^{t}I_{0})$, with $t\geq 1$ such that $t\equiv p-1 \,\,(\mod\,p)$ and $t\equiv           q-1\,\,(\mod\,q)$.
\item $(\tau^{t-p+1}I_{q-1},F,G,\tau^{t}I_{p+q-1})$, with $t\geq 1$ such that $t\equiv p-1\,\,(\mod\,p)$ and $t\equiv        q-1 \,\,(\mod \,q)$.
\item $(\tau^{t-r} I_{p+q-r-2},F,G,\tau^{t}I_{r+1})$, with $t\geq 1$ such that $t\equiv r\,\,(\mod\,p)$,
      $0<r<p-1$ e $t\equiv q-1\,\,(\mod\,q)$.
\item $(X,F,G,\tau^{t}I_{p+r})$, with $t\geq 1$ such that $t\equiv r\,\,(\mod\,q)$, $0<r<q-1$ and $t\equiv p-1\,(\mod\,p)$, where 
      \begin{displaymath}
      \begin {array}{rll}
      X \cong &\left \{ \begin {array}{ll}
      \tau^{t-r}I_{p-(r+1)}, \textrm{ if } r < p\\
      \tau^{t-(p-1)}I_{r}, \textrm{ if } r \geq p.\\
      \end{array} \right.\\
      \end{array}
      \end{displaymath}
\end{enumerate}

\end{theorem}

\section*{Acknowledgments}
The first named author thanks CAPES and CNPq for financial support, during her PhD. The authors thank Maria Izabel Ramalho Martins for helpful discussions. The first named author also thanks her friend  Heily Wagner. The second author thanks CNPq for the research grant. Both authors used a tematic grant from Fapesp, to visit Marcelo Lanzilotta in Universidad de la Rep\'ublica, Uruguay, they thank Fapesp for support and Marcelo for the warm hospitality.

\end{document}